\documentclass{article}

\usepackage{amsmath}
\usepackage{amssymb}
\usepackage{hyperref}
\usepackage{amssymb}
\usepackage{amsthm}

\theoremstyle{plain}
\newtheorem{theorem}{Theorem}[section]
\newtheorem{proposition}[theorem]{Proposition}
\newtheorem{lemma}[theorem]{Lemma}

\theoremstyle{remark}
\newtheorem{remark}[theorem]{Remark}

\begin{document}
\title{Dirac brackets and
reduction of invariant bi-Poisson structures}
\author{Ihor V. Mykytyuk  \\   Andriy Panasyuk}

\maketitle

\begin{abstract}
Let $X$ be a manifold with a bi-Poisson structure $\{\eta^t\}$ generated by a
pair of $G$-invariant symplectic structures $\omega_1$ and
$\omega_2$, where the Lie group
$G$ acts properly on $X$. Let $H$
be some isotropy subgroup for this action representing the
principle orbit type and
$X^r_\mathfrak{h}$ be the submanifold of
$X$ consisting of the points in
$X$ with the stabilizer algebra equal to the Lie algebra
$\mathfrak{h}$ of
$H$ and with the stabilizer group conjugated to $H$ in
$G$. We prove that the pair of symplectic structures
$\omega_1|_{X^r_\mathfrak{h}}$ and
$\omega_2|_{X^r_\mathfrak{h}}$ generates an
$N(H^0)/H^0$-invariant bi-Poisson structure on
$X^r_\mathfrak{h}$, where $N(H^0)$ is the normalizer in
$G$ of the identity component $H^0$ of
$H$. The action of $\widetilde G=N(H^0)/H^0$ on
$X^r_\mathfrak{h}$ is locally free and proper and, moreover, the spaces
$A^G$ of $G$-invariant functions on $X$ and
$A^{\widetilde G}$ of $\widetilde G$-invariant functions on
$X^r_\mathfrak{h}$ can be canonically identified and therefore
the bi-Poisson structure
$\{(\eta^t)'\}$ induced on
$A^G\simeq A^{\widetilde G}$ can be treated as the reduction
with respect to a {\em locally free} action of
a Lie group which essentially simplifies the study of
$\{(\eta^t)'\}$.

\end{abstract}

\section{Introduction}

Two Poisson structures
$\eta_1$ and $\eta_2$ are said to be compatible if the sum
$\eta_1+\eta_2$, or, equivalently, any linear combination
$\eta^t=t_1\eta_1+t_2\eta_2$ is a Poisson structure. The
family
$\{\eta^t\}$ is called a bi-Poisson structure. In this paper
we consider the problem of reduction of a bi-Poisson structure
$\{\eta^t\}$, which is generated by two symplectic
structures,
$G$-invariant with respect to a proper action of a Lie group
$G$ on a connected manifold $X$ to a bi-Poisson structure
$\{\widetilde\eta^t\}$ being
$\widetilde G$-invariant with respect to a proper
{\it locally free} action of some Lie group
$\widetilde G$ on some submanifold
$\widetilde X\subset X$ (Theorem~\ref{th.4}).
We consider a particular case when the spaces
$A^G$ of $G$-invariant functions on $X$ and
$A^{\widetilde G}$ of $\widetilde G$-invariant functions on
$\widetilde X$ can be canonically identified and therefore
the bi-Poisson structure
$\{(\eta^t)'\}$ induced on
$A^G\simeq A^{\widetilde G}$ can be treated as the reduction
with respect to a {\em locally free} action of
a Lie group which essentially simplifies the study of
$\{(\eta^t)'\}$.

Note that, given a symplectic form
$\omega$ on a manifold $X$ and a symplectic submanifold
$\widetilde X \subset X$, the Poisson bracket related to the
Poisson structure $\widetilde\eta=(\omega|_{\widetilde X})^{-1}$
is an example of the so-called Dirac bracket \cite[Sect.~8.5]{MR}.
In general, if two Poisson structures
$\omega_1^{-1}$ and
$\omega_2^{-1}$ are compatible, the Dirac brackets
$(\omega_1|_{\widetilde X})^{-1},(\omega_2|_{\widetilde X})^{-1}$
need not be so (here
$\widetilde X$ is a symplectic submanifold with respect to
both $\omega_1,\omega_2$). In this paper we deal with a very
special situation, when
$(\omega_1|_{\widetilde X})^{-1},(\omega_2|_{\widetilde X})^{-1}$
are compatible, which is a consequence of the
$G$-invariance and the special choice of the submanifold
$\widetilde X$. In more detail this situation can be
described as follows.

Given a proper action of a connected Lie group
$G$ on a connected manifold $X$ and an isotropy subgroup
$H\subset G$ representing the principle orbit type,
consider the subset
$X_{(H)}$ of $X$ consisting of the points in
$X$ with the stabilizer conjugated to $H$ in $G$.
Since the manifold $X$ is connected the subset
$X_{(H)}$ is connected, open, and  dense
in $X$ (\cite{DK}). We show that the subset
$X^r_\mathfrak{h}$ of $X_{(H)}$ consisting of the points in
$X_{(H)}$ with the stabilizer algebra precisely equal to
$\mathfrak{h}$, the Lie algebra of $H$,
is a smooth embedded submanifold of $X$. The subgroup
$N(H^0)\subset G$, the normalizer group in
$G$ of the identity component $H^0$ of $H$, acts on
$X^r_\mathfrak{h}$ and the action of the quotient group $N(H^0)/H^0$ on
$X^r_\mathfrak{h}$ is {\em locally free} and proper.

Let $\omega_1$ and
$\omega_2$ be two $G$-invariant symplectic structures on
$X$ determining a bi-Poisson structure
$\{\eta^t\}$ on $X$. We show that
$X^r_\mathfrak{h}$ is a symplectic submanifold for arbitrary
$G$-invariant symplectic structure on
$X$, in particular, the forms
$\omega_1|_{X^r_\mathfrak{h}}$ and
$\omega_2|_{X^r_\mathfrak{h}}$ are nondegenerate. Then we prove that the pair of symplectic structures
$\omega_1|_{X^r_\mathfrak{h}}$ and
$\omega_2|_{X^r_\mathfrak{h}}$ generates a
$N(H^0)/H^0$-invariant bi-Poisson structure
$\{\widetilde\eta^t\}$ on
$X^r_\mathfrak{h}$ (Theorem~\ref{th.4}). Due to the isomorphism
$X_{(H)}/G=X^r_\mathfrak{h}/(N(H^0)/H^0)$,
the second quotient space is a smooth manifold. As a result the sets
$A^G$ of $G$-invariant functions on $X_{(H)}$ and
$A^{\widetilde G}$ of $\widetilde G$-invariant functions on
$X_{\mathfrak{h}}$, where
$\widetilde G=N(H^0)/H^0$, can be canonically identified.
The bi-Poisson structures
$\{\eta^t\}$ and
$\{\widetilde\eta^t\}$ restricted to the space
$A^G\simeq A^{\widetilde G}$ determine the same bi-Poisson
structure (Theorem~\ref{th.4}).

Since the proper action of the group
$N(H^0)/H^0$ on the manifold $X_\mathfrak{h}$
is locally free, for investigation of the algebraic properties of the bi-Poisson algebra
$A^G\simeq A^{\widetilde G}$
we can use methods developed in the paper~\cite{Pa} for
locally free actions.
Roughly, such an investigation consists of two parts: first is based on hamiltonicity of the action of the corresponding Lie group with respect to a generic Poisson structure of the pencil; second is the study of certain exceptional representatives of the pencil. On the first stage, once the hamiltonicity is established one can use the so-called inertia lemma from the theory of hamiltonian actions \cite[Lemma 2.1]{GS1} (relating the image of the moment map to the stabilizer of the action)  for calculating the rank of the reduction of the generic Poisson structure at a generic point.  Under the assumption  that the action of the corresponding group is locally free this lemma says that the corank of the reduction of the generic Poisson structure is independent on the parameter of the pencil and equal to the index of the Lie algebra of $G$. Without this assumption the situation is much more complicated. This is the reason, why our passage from non locally free to locally free actions is crucial.

We illustrate the theory by a class of examples of reductions of bi-Poisson structures on cotangent
bundles to coadjoint orbits (homogeneous spaces)
$G/K$, where a compact Lie group $G$ acts on $G/K$
and then on the cotangent bundle $T^*(G/K)$ by the lifted action, see
Section \ref{s.3}. Here
$\omega_1$ is the canonical symplectic form
$\Omega$ on the cotangent bundle and
$\omega_2$ is equal to the sum of $\Omega$
and the pull-back of the Kirillov--Kostant--Souriau form. In particular,  we  describe the submanifolds
$X_{(H)}$ and $X^r_\mathfrak{h}$ and the reduced bi-Poisson structure $\{\widetilde\eta^t\}$ on
$X^r_\mathfrak{h}$ (see Proposition \ref{pr.6} and its proof).

These examples first appeared in our paper \cite{MP} (note that the proofs of the results of Section \ref{s.3} are new and independent of that from \cite{MP}), where they served
as a tool in the proof of the complete integrability for
geodesic flows of some metrics. The present paper arose from our attempt of understanding the general principle standing behind the examples mentioned. The results of
this paper are also intended as a tool which can be effectively applied to the study of complete integrability of similar systems, however, such a study lies beyond of the scope of this parer since we hope that the results of the present paper are of interest on their own.

The paper is organized as follows. It is divided to three
sections and Appendix among which Section \ref{s.2} is the principal one. The main result of the paper, Theorem \ref{th.4}, is contained in Subsection \ref{s.2.3}, while Subsections \ref{ss.2.1}--\ref{ss.2.2} are intended as introductory ones. They contain definitions and  general results needed for the formulation and proof of our main result (its crucial ingredients are Lemmas \ref{le.2} and \ref{th.3}). Section \ref{s.3} contains the above mentioned examples. In Appendix we formulate and prove one statement (Lemma \ref{le.8}) from the general theory of Lie groups which is used in the introductory considerations of Subsection \ref{ss.2.2}.

\section{Proper actions of Lie groups and reductions
of invariant bi-Poisson structures}
\label{s.2}

Let $G$ be a connected Lie group acting properly on a smooth
connected manifold $X$. For any point $x\in X$ denote by
$G_x$ its isotropy group. Remark that the group
$G_x$ is compact because the action of $G$ on
$X$ is proper.

Fix some isotropy subgroup $H\subset G$ determining the
{\em principal orbit type}. In this case
the subset
\begin{equation}\label{eq.1}
X_{(H)}=\{x\in X: G_x=gHg^{-1}\ \text{for some}\ g\in G\}
\end{equation}
of $X$, consisting of all orbits $G\cdot x$ in
$X$ isomorphic to $G/H$, is an open and
dense subset of $X$ (see~\cite[\S 2.8 and Th. 2.8.5]{DK}).
The open submanifold $X_{(H)}\subset X$ is
$G$-invariant by definition.
It is well known that the orbit space
$X_{(H)}/G$ is a smooth manifold.
Mainly to fix the notation we shall prove this fact below.

Consider the subset
\begin{equation}\label{eq.2}
X_H=\{x\in X: G_x=H\}
\end{equation}
of $X$ consisting of the points in $X$ with stabilizer precisely equal to $H$.
It is clear that $X_H\subset X_{(H)}$.
The set $X_H$ is a smooth embedded submanifold of
$X$~\cite[Prop. 2.4.7]{OR}.
Let $N(H)$ be the normalizer group of $H$ in
$G$. It is easy to see that the subgroup
$N(H)$ acts on $X_H$ and that every $G$-orbit in
$X_{(H)}$ intersects $X_H$ on an
$N(H)$ orbit. Furthermore, the quotient group
$N(H)/H$ acts freely on
$X_H$ and generates the same orbit space. This action of
$N(H)/H$ is proper because the subgroup
$N(H)\subset G$ is closed. Therefore
$X_H/(N(H)/H)$ is a smooth manifold~\cite[Ch. 3, \S 1.5, Prop. 10]{Bou}
and, consequently, due to orbit isomorphism
\begin{equation}\label{eq.3}
X_{(H)}/G\simeq X_H/(N(H)/H)
\end{equation}
$X_{(H)}/G$ is also a smooth manifold.

\subsection{\tt The submanifold $X^r_\mathfrak{h}$ of
the single orbit type \\ submanifold ${X_{(H)}}$}
\label{ss.2.1}

Let ${\mathfrak{g}}$ be the Lie algebra of the Lie group
$G$.
Let $\mathfrak{h}$ and
$\mathfrak{n}(H)$ be the Lie algebras of the Lie groups
$H$ and $N(H)$ respectively. The algebra
$\mathfrak{n}(H)$ is a subalgebra of the normalizer
$\mathfrak{n}(\mathfrak{h})$ of the algebra
$\mathfrak{h}$ in ${\mathfrak{g}}$ and coincides with
$\mathfrak{n}(\mathfrak{h})$ if the Lie group
$H$ is connected. In general
$\mathfrak{n}(H)\ne \mathfrak{n}(\mathfrak{h})$.
Since the Lie subgroup
$\operatorname{Ad} (H)$ of
$\operatorname{Ad}(G)$ is compact, there is an
$\operatorname{Ad}(H)$-invariant scalar product
$\langle\cdot ,\cdot \rangle^H$ on the Lie algebra
${\mathfrak{g}}$. Denote by
$\mathfrak{p}$ the orthogonal complement to
$\mathfrak{n}(\mathfrak{h})$ in
${\mathfrak{g}}$ with respect to
$\langle\cdot ,\cdot \rangle^H$.
Due to the connectedness of the group
$H^0$ with the Lie algebra $\mathfrak{h}$ we have
\begin{equation}\label{eq.4}
{\mathfrak{g}}=\mathfrak{p}\oplus\mathfrak{n}(\mathfrak{h}),
\quad
\operatorname{Ad}(H^0)(\mathfrak{n}(\mathfrak{h}))
=\mathfrak{n}(\mathfrak{h}),
\quad
\operatorname{Ad}(H^0)(\mathfrak{p})=\mathfrak{p}.
\end{equation}

As we remarked above in general
$\mathfrak{n}(H)\ne \mathfrak{n}(\mathfrak{h})$.
Therefore it is more useful from the point of view of
calculations (see example in Section~\ref{s.3}) to consider
also the subset of $X$
\begin{equation}\label{eq.5}
X_\mathfrak{h}=\{x\in X: {\mathfrak{g}}_x=\mathfrak{h}\}=
\{x\in X: G_x^0=H^0\},
\end{equation}
where ${\mathfrak{g}}_x$ stands for the Lie algebra of the
isotropy group
$G_x$ and $A^0$ for the connected component of the identity
element (the identity component for short) of a Lie group
$A\subset G$. In general,
$X_\mathfrak{h}\not\subset X_{(H)}$. Therefore
we will consider the subset
\begin{equation}\label{eq.6}
X^r_\mathfrak{h}=X_\mathfrak{h}\cap X_{(H)}
\end{equation}
of $X_{(H)}$ (of ``regular'' points).

Clearly, $X^r_\mathfrak{h}$ contains
$X_H$. We will prove below that
$X^r_\mathfrak{h}$ is an embedded submanifold of
$X_{(H)}$ and
$\dim X^r_\mathfrak{h}-\dim
X_H=\dim\mathfrak{n}(\mathfrak{h})-\dim\mathfrak{n}(H)$.

The normalizer $N(H^0)$ of $H^0$ in $G$ coincides with the
normalizer
$$
N(\mathfrak{h})=\{g\in G:
\operatorname{Ad}(g)(\mathfrak{h})=\mathfrak{h}\}
$$
of $\mathfrak{h}$ in $G$. The Lie algebra $\mathfrak{n}(\mathfrak{h})$
is the Lie algebra of $N(H^0)=N(\mathfrak{h})$.

\begin{lemma}\label{le.1}
The set $X^r_\mathfrak{h}$ is embedded submanifold of
the manifold $X_{(H)}$ and for any $x\in X^r_\mathfrak{h}$
the tangent space $T_x X^r_\mathfrak{h}$ is given by
\begin{equation}\label{eq.7}
T_x X^r_\mathfrak{h}=\{v\in T_x X_{(H)}:
h_{*}(x)(v)=v, \forall h\in H^0\}.
\end{equation}
The quotient group $N(H^0)/H^0$ acts locally freely on the set
$X^r_\mathfrak{h}$ and
\begin{equation}\label{eq.8}
X^r_\mathfrak{h}/(N(H^0)/H^0)\simeq X_{(H)}/G
\end{equation}
is a smooth manifold.
\end{lemma}
\begin{proof}
To prove the lemma we will use the method
from~\cite[Ch.2,\S 2.3,\S 2.4]{OR}. Due to the fact that
$X_{(H)}$ is a single orbit type manifold the local
description of the
$G$-action on this connected manifold is very simple. For
the point $x\in X_{(H)}$ there is a
$G$-invariant open neighborhood $O(x)$ in
$X_{(H)}$ and a $G$-equivariant diffeomorphism
$\phi: G/G_x\times W\to O(x)$, where
$G$ acts naturally on $G/G_x$ and trivially on
$W$~\cite[Th.2.3.28]{OR}. Here the cross-section
$W$ is an open ball around $0$ in some real linear space
${\mathbb R}^k$ (of dimension
$k=\dim X-\dim(G/G_x)$) and $\phi(o,0)=x$, where
$o=G_x\in G/G_x$.

Since $x\in X^r_\mathfrak{h}\subset X_{(H)}$ we have
$G_x^0=H^0$. Under the above mentioned
$\phi$-identification of the open neighborhood
$O(x)$, $x\in X^r_\mathfrak{h}$, with $G/G_x\times W$ the subset
$O_\mathfrak{h}(x)=O(x)\cap X^r_\mathfrak{h}$ is $\phi$-isomorphic to
$N(H^0)/G_x\times W\subset G/G_x\times W$, where
$N(H^0)/G_x$ is considered as a closed embedded submanifold
of $G/G_x$~\cite[Prop.2.4.6]{OR}. From this local description
it follows that
$X^r_\mathfrak{h}$ is a (locally closed) embedded submanifold of
$X_{(H)}$. The submanifold $O_\mathfrak{h}(x)$ is
$N(H^0)$-invariant.

Let us prove relation~(\ref{eq.7}).
The group $H^0$ acts on
$W$ trivially and the tangent action of
$H^0$ on the tangent space
${\mathfrak{g}}/\mathfrak{h}=T_o(G/G_x)$ is induced by the
$\operatorname{Ad}(H^0)$-action on
${\mathfrak{g}}$. Let $\xi\in{\mathfrak{g}}$.
By formula~(\ref{eq.25}) from Appendix
$\operatorname{Ad}(h)(\xi+\mathfrak{h})=\xi+\mathfrak{h}$
for all $h\in H^0$ if and only if
$\xi\in\mathfrak{n}(\mathfrak{h})$. Taking into account that
$\mathfrak{n}(\mathfrak{h})/\mathfrak{h}=T_o(N(H^0)/G_x)$,
we obtain that
$$
T_{(o,0)}(N(H^0)/G_x\times W)=\{v\in T_{(o,0)}(G/G_x\times W):
h_{*}(o,0)(v)=v, \forall h\in H^0\}.
$$
Since the diffeomorphism $\phi$ is $G$--equivariant,
we get~(\ref{eq.7}).

Since for each $x\in X^r_\mathfrak{h}\subset X_{(H)}$
its isotropy group $G_x$
is conjugated to $H$ in $G$, it is easy to check that
\begin{list}{}{\listparindent 0pt \itemsep 0pt
\parsep  0pt \topsep 0pt}
\item[1)]
the subgroup $N(H^0)$ acts on $X^r_\mathfrak{h}$ and
$H^0\subset N(H^0)$ acts trivially on $X^r_\mathfrak{h}$;
\item[2)]
every $G$-orbit in $X_{(H)}$ intersects
$X^r_\mathfrak{h}$ by an $N(H^0)$-orbit;
\item[3)]
$N(H^0)\cdot X_H= X^r_\mathfrak{h}$ (if $G_x^0=H^0$
and $gG_xg^{-1}=H$, then $g\in N(H^0)$).
\end{list}
The quotient group $N(H^0)/H^0$ acts {\it locally freely} on
$X^r_\mathfrak{h}$ (with finite isotropy group $G_x/H^0\simeq H/H_0$
at $x\in X^r_\mathfrak{h}$) and generates the same orbit space as
$G$ on $X_{(H)}$. This action of
$N(H^0)/H^0$ is proper because the subgroup
$N(H^0)\subset G$ is closed. Since by relation~(\ref{eq.3})
$X_{(H)}/G$ is a smooth manifold, the quotient space
$X^r_\mathfrak{h}/(N(H^0)/H^0)\simeq X_{(H)}/G$
is also a smooth manifold.
\qed
\end{proof}

\subsection{\tt The local structure of the single orbit type \\
submanifold $\mathbf{X_{(H)}}$ near $\mathbf{X^r_\mathfrak{h}}$}
\label{ss.2.2}

 The action of $G$ defines a linear map
$\xi\mapsto \xi_X$, where
$\xi_X$ denotes the vector field on
$X$ generated by one-parameter subgroup
$\exp t\xi\subset G$. For any subspace
$\mathfrak{a}\subset{\mathfrak{g}}$ and point $x\in X$ put
$\mathfrak{a}(x){\stackrel{\mathrm{def}}{=}}\{\xi_X(x):
\xi\in\mathfrak{a}\}$.

In this subsection we describe a canonical complementary
subbundle ${\mathcal{P}}$ to $ TX^r_\mathfrak{h}$
in $TX_{(H)}|_{X^r_\mathfrak{h}}$ for which the splitting
$TX_{(H)}|_{X^r_\mathfrak{h}}={\mathcal{P}}\oplus TX^r_\mathfrak{h}$
is orthogonal with respect to arbitrary $G$-invariant nondegenerate
form on $X$. The existence of a such canonical subbundle ${\mathcal{P}}$
determines a local structure of $X_{(H)}$
near $X^r_\mathfrak{h}$.

Choose some point $x\in X^r_\mathfrak{h}\subset X_{(H)}$.
Due to the compactness of the Lie subgroup
$\operatorname{Ad} (G_x)$ of
$\operatorname{Ad}(G)$ there exists an
$\operatorname{Ad}(G_x)$-invariant scalar product
$\langle\cdot ,\cdot \rangle^{G_x}$ on the Lie algebra
${\mathfrak{g}}$. Denote by
$\mathfrak{p}^x$ the orthogonal complement to
$\mathfrak{n}(\mathfrak{h})$ in
${\mathfrak{g}}$ with respect to
$\langle\cdot ,\cdot \rangle^{G_x}$. Since
$H^0=G_x^0$, we have
\begin{equation}\label{eq.9}
{\mathfrak{g}}=\mathfrak{p}^x\oplus\mathfrak{n}(\mathfrak{h}),
\quad
\operatorname{Ad}(H^0)(\mathfrak{n}
(\mathfrak{h}))=\mathfrak{n}(\mathfrak{h}),
\quad
\operatorname{Ad}(H^0)(\mathfrak{p}^x)=\mathfrak{p}^x.
\end{equation}
Identifying the tangent space to the homogeneous space
$G/G_x$ at $o=G_x\in G/G_x$ with the orthogonal complement
$\mathfrak{h}^{\bot x}\subset{\mathfrak{g}}$ to
$\mathfrak{h}$ in ${\mathfrak{g}}$ with respect to
$\langle\cdot ,\cdot \rangle^{G_x}$, we obtain that
$\mathfrak{p}^x\subset\mathfrak{h}^{\bot x}$
is a complementary subspace to the tangent space
$T_o(N(H^0)/G_x)$ in $T_o(G/G_x)$. Using our
$G$-equivariant identification
$\phi: G/G_x\times W\to O(x)$,
$\phi(o,0)=x$, we conclude that the space
$\mathfrak{p}^x(x){\stackrel{\mathrm{def}}{=}}\{\xi_X(x),\
\xi\in\mathfrak{p}^x\}$
is a complementary subspace to
$T_x X^r_\mathfrak{h}$ in $T_x X_{(H)}$.

The $\operatorname{Ad}(G_x)$-invariant scalar product
$\langle\cdot ,\cdot \rangle^{G_x}$ on
$T_o(G/G_x)=\mathfrak{h}^{\bot x}$ and any scalar product on
$T_0 W$ determine an
$G_x$-invariant scalar product on the tangent space
$T_o(G/G_x)\oplus T_0 W$. Now using the
$G$-equivariant diffeomorphism $\phi$ we get the
$G_x$-invariant scalar product
$\langle\cdot ,\cdot \rangle_x$ on the space
$T_x X_{(H)}$ at $x=\phi(o,0)$ such that
$\langle \mathfrak{p}^x(x) , T_x X^r_\mathfrak{h}\rangle_x=0$.
In general, $\mathfrak{p}^x\ne\mathfrak{p}$ (the subspace
$\mathfrak{p}$ was defined in Subsection \ref{ss.2.1}) but
\begin{equation}\label{eq.10}
\mathfrak{p}^x\oplus\mathfrak{h}=\mathfrak{p}\oplus\mathfrak{h}.
\end{equation}
The proof of this identity is given  in
Appendix (see Lemma~\ref{le.8}). Now, taking into account that
$\mathfrak{h}$ is the isotropy algebra of the point
$x\in X^r_\mathfrak{h}$, i.e. $\mathfrak{h}(x)=0$, we obtain that
$\mathfrak{p}^x(x)=\mathfrak{p}(x)$. Thus the space
\begin{equation}\label{eq.11}
{\mathcal{P}}(x){\stackrel{\mathrm{def}}{=}}\mathfrak{p}(x)
=\{\xi_X(x),\ \xi\in\mathfrak{p}\},\quad x\in X^r_\mathfrak{h},
\end{equation}
is the orthogonal complement to the tangent space
$T_x X^r_\mathfrak{h}$ in $T_x X_{(H)}$:
\begin{equation}\label{eq.12}
T_xX_{(H)}={\mathcal{P}}(x)\oplus T_xX^r_\mathfrak{h},\quad
\langle {\mathcal{P}}(x) , T_x X^r_\mathfrak{h}\rangle_x=0,\quad
x\in X^r_\mathfrak{h}.
\end{equation}
We will show below that the space ${\mathcal{P}}(x)$ is independent
of the choice of the $G_x$-invariant scalar
products $\langle\cdot ,\cdot\rangle^{G_x}$ on ${\mathfrak{g}}$
and $\langle\cdot ,\cdot\rangle_x$ on $T_x X_{(H)}$.

Since $h_{*}(\xi_X)=(\operatorname{Ad}(h)(\xi))_X$ for any vector
$\xi\in{\mathfrak{g}}$ and
$\operatorname{Ad}(H^0)(\mathfrak{p}) =\mathfrak{p}$,
the space ${\mathcal{P}}(x)$, $x\in X^r_\mathfrak{h}$ is
$H^0$-invariant. It is evident that the union
${\mathcal{P}}=\bigcup_{x\in X^r_\mathfrak{h}}{\mathcal{P}}(x)$
is a trivial vector bundle over
$X^r_\mathfrak{h}$ and
$TX_{(H)}|_{X^r_\mathfrak{h}}={\mathcal{P}}\oplus TX^r_\mathfrak{h}$.
The vector fields $\xi_X|_{X^r_\mathfrak{h}}$,
$\xi\in\mathfrak{p}$, are global sections of
${\mathcal{P}}$.

\begin{lemma}\label{le.2}
Let $\alpha(x)$ be an
$G_x$-invariant nondegenerate bi-linear form on the space
$T_x X_{(H)}$, $x\in X^r_\mathfrak{h}$. Then
$\alpha(x)({\mathcal{P}}(x),T_x X^r_\mathfrak{h})=0$, i.e.
${\mathcal{P}}(x)$ is the orthogonal complement to the space
$T_x X^r_\mathfrak{h}$ in $T_x X_{(H)}$ with respect to the form
$\alpha(x)$ and the restrictions $\alpha(x)|_{{\mathcal{P}}(x)}$,
$\alpha(x)|_{T_x X^r_\mathfrak{h}}$ are nondegenerate.
\end{lemma}
\begin{proof}
To prove that $\alpha(x)({\mathcal{P}}(x),T_x X^r_\mathfrak{h})=0$
we will use the method of the proof of Lemma 27.1
in~\cite{GS}. We have shown that there exists a
$G_x$-invariant scalar product
$\langle\cdot,\cdot\rangle_x$ on the space
$T_x X_{(H)}={\mathcal{P}}(x)\oplus T_x X^r_\mathfrak{h}$ such that
formula~(\ref{eq.12}) holds. The form
$\alpha(x)$ is
$G_x$-invariant with respect to the tangent action
$h_{*}(x): T_x X_{(H)}\to T_x X_{(H)}$ of the group
$G_x$. Thus there exists a unique nondegenerate linear map
$J:T_x X_{(H)}\to T_x X_{(H)}$ such that
$\alpha(x)(u,v)=\langle u,Jv \rangle_x$ for all
$u,v\in T_x X_{(H)}$ and
$J\cdot h_{*}(x)=h_{*}(x)\cdot J$ for all
$h\in G_x\supset H^0$. By~(\ref{eq.7}) the subspace
$T_x X^r_\mathfrak{h}\subset T_x X_{(H)}$ is the set of all
$H^0$-fixed vectors in
$T_x X_{(H)}$. Now we get the inclusion
$J(T_x X^r_\mathfrak{h})\subset T_x X^r_\mathfrak{h}$
due to the fact that $J$ commutes with the $H^0$-action on
$T_x X_{(H)}$. Thus
$\alpha(x)({\mathcal{P}}(x),T_x X^r_\mathfrak{h})=\langle
{\mathcal{P}}(x),JT_x X^r_\mathfrak{h} \rangle_x=0$
by~(\ref{eq.12}). Since
$T_x X_{(H)}={\mathcal{P}}(x)\oplus T_x X^r_\mathfrak{h}$
and the form
$\alpha(x)$ is nondegenerate, we obtain the last assertion
of the lemma. \qed
\end{proof}

The following lemma asserts the existence
of local coordinate systems in $X_{(H)}$
near the submanifold $X^r_\mathfrak{h}$
consistent with any $G$-invariant
nondegenerate bi-linear form on $X$.

\begin{lemma}\label{th.3}
For each point $x\in X^r_\mathfrak{h}$
there exists an open subset $U(x)\subset X_{(H)}$
and a coordinate
system $(U(x), y_1,\ldots,y_p,y_{p+1},\ldots,y_{m})$ in
$X_{(H)}$ around the point $x$ such that
\begin{list}{}{\listparindent 0pt \itemsep 0pt
\parsep  0pt \topsep 0pt}
\item{$1)$}
all coordinates of the point
$x$ vanish: $y_1(x)=\ldots=y_m(x)=0$;
\item{$2)$}
the subset
$U_\mathfrak{h}(x)=U(x)\cap X^r_\mathfrak{h}$ of $U(x)$ is the set
$\{z\in U(x): y_1(z)=\ldots=y_p(z)=0\}$;
\item{$3)$}
the vectors
${\partial}/{\partial y_i}$, $i=1,\ldots,p$, and the vectors
${\partial}/{\partial y_j}$, $j=p+1,\ldots,m$, at a point
$z\in U_\mathfrak{h}(x)$ span the spaces ${\mathcal{P}}(z)$ and
$T_z X^r_\mathfrak{h}$ respectively.
\item{$4)$}
any $G$-invariant nondegenerate bi-linear form
$\alpha$ on $X$ at a point $z\in U(x)$
in the corresponding basis $\{{\partial}/{\partial y_1},\ldots,{\partial}/{\partial y_m}\}$, has the matrix \linebreak {\small $\left(
\begin{array}{cc}
 A(y(z))  & B(y(z)) \\
C(y(z))   & D(y(z))
\end{array}\right)$}
such that $B(y(z))=C(y(z))=0$ and the matrices $A(y(z))$,
$D(y(z))$ are nondegenerate for
$z\in U_\mathfrak{h}(x)$.
\end{list}
\end{lemma}
\begin{proof}
Recall that the group
$N(H^0)$ is a closed subgroup of $G$ because
$H^0$, as the identity component of the closed subgroup
$H\subset G$, is also closed in $G$. Also we have the
$\operatorname{Ad}(H^0)$-invariant splitting
${\mathfrak{g}}=\mathfrak{p}\oplus\mathfrak{n}(\mathfrak{h})$
of ${\mathfrak{g}}$ (see formula~(\ref{eq.4})). Therefore for
some open $\operatorname{Ad}(H^0)$-invariant ball $Y$ around
$0$ in $\mathfrak{p}$ the map
$$
Y\times N(H^0)\to G,
\qquad (y,n)\mapsto \exp y \cdot n,
$$ is an $H^0$-equivariant diffeomorphism onto some
open neighborhood of the identity element in $G$.
This map intertwines the action
$h\cdot(y,n)=(\operatorname{Ad}(h)(y), hn)$ of $H^0$
on $Y\times N(H^0)$ and the left action of $H^0$ on $G$.
Thus the map
$$
Y\times N(H^0)/G_x\to G/G_x,
\quad
(y,nG_x)\mapsto (\exp y \cdot n)G_x
$$
is an $H^0$-equivariant diffeomorphism onto
open neighborhood of the point $o=G_x$ in $G/G_x$
and, consequently, the map
$$
Y\times N(H^0)/G_x\times W \to \phi(G/G_x\times W)=O(x),
\ (y,nG_x, w)\mapsto \phi(\exp y\cdot nG_x,w)
$$
is an $H^0$-equivariant diffeomorphism onto some open
$H^0$-invariant neighborhood
$O_1(x)\subset O(x)$ in $X_{(H)}$ containing
a neighborhood
$O_\mathfrak{h}(x)=\phi(N(H^0)/G_x\times W)$
of $x$ in $X^r_\mathfrak{h}$.
Here the action of $H^0$ on
$Y\times N(H^0)/G_x\times W$ is induced by the action of
$H^0$ on $Y\times N(H^0)$, i.e.
$h\cdot(y,nG_x,w)=(\operatorname{Ad}(h)(y), hnG_x,w)$ for
$h\in H^0$. By the $G$-equivariance of
$\phi$, the map
$$
\psi:Y\times O_\mathfrak{h}(x) \to O_1(x),\quad
(y,z)\mapsto (\exp y)\cdot z
$$
is also a diffeomorphism such that
$\psi(0,z)=z$ for all $z\in O_\mathfrak{h}(x)$.
Moreover, $\psi_{*(0,z)}(T_0Y,0)={\mathcal{P}}(z)$
for $z\in O_\mathfrak{h}(x)$ because by~(\ref{eq.11})
${\mathcal{P}}(z)=\{\xi_X(x),\ \xi\in\mathfrak{p}\}$
and $Y\subset\mathfrak{p}$.
This diffeomorphism $\psi$ is
$H^0$-equivariant with respect to the action
$h\cdot (y,z)=(\operatorname{Ad}(h)(y), h\cdot z)$ of
$H^0$ on $Y\times O_\mathfrak{h}(x)$ and the
$H^0$-action on $O_1(x)\subset X_{(H)}$.

The existence of the diffeomorphism
$\psi$ means in particular that there exists a coordinate
system $(U(x), y_1,\ldots,y_p,y_{p+1},\ldots,y_{m})$ in
$X_{(H)}$ around the point
$x\in X^r_\mathfrak{h}$ with properties $1)-3)$.

Let us prove property $4)$ for this coordinate system.
Since $h\cdot x=x$ for all
$h\in G_x$, the nondegenerate form $\alpha(x)$ is
$G_x$-invariant with respect to the tangent action
$h_{*}(x): T_x X\to T_x X$ of the group $G_x$. Then
$\alpha(x)({\mathcal{P}}(x),T_x X^r_\mathfrak{h})=0$
and
the restrictions $\alpha(x)|_{{\mathcal{P}}(x)}$
$\alpha(x)|_{T_x X^r_\mathfrak{h}}$ are nondegenerate in view of Lemma~\ref{le.2}.
Therefore by property 3) of the coordinate system
under consideration
around the point $x$  the matrices
$B(y(z))$, $C(y(z))$ vanish and
the matrices $A(y(z))$, $D(y(z))$
are nondegenerate for any $z\in U_\mathfrak{h}(x)$.
\qed
\end{proof}

\subsection{\tt The principal orbit type submanifold ${X_{(H)}}$
and \\ reduced bi-Poisson structures on ${X^r_\mathfrak{h}}$}
\label{s.2.3}

Here as before  $X_{(H)}$ is a principal orbit type submanifold of
$X$. We will use the notation introduced in the previous
subsections. Denote by
$\mathcal{E}(M)$ the space of smooth functions on a manifold
$M$.

Let $\eta$ be an $G$-invariant Poisson structure
on the manifold $X$. Put
$A^G\subset \mathcal{E}(X_{(H)})$ for the set of all
$G$-invariant functions on the open submanifold
$X_{(H)}\subset X$. By the $G$-invariance of
$\eta$, the space $A^G$ is a Poisson subalgebra of
$(\mathcal{E}(X_{(H)}),\eta)$. The structure
$\eta$ determines a Poisson structure
on the smooth manifold (see~(\ref{eq.8}))
$$
\mathbf X=X_{(H)}/G\simeq X^r_\mathfrak{h}/(N(H^0)/H^0)
$$
and $A^{G}\simeq\mathcal{E}(\mathbf X)$.
Put $\widetilde G= N(H^0)/H^0$. Denoting by
$\pi_{(H)}: X_{(H)}\to \mathbf X$ and
$\pi_{\mathfrak{h}}: X^r_{\mathfrak{h}}\to \mathbf X$
the natural submersions, we obtain two isomorphic Poisson algebras,
$A^G=\pi_{(H)}^*(\mathcal{E}(\mathbf X))$ of
$G$-invariant functions on $X_{(H)}$ and
$A^{\widetilde G}=\pi_{\mathfrak{h}}^*(\mathcal{E}(\mathbf X))$ of
$\widetilde G$-invariant functions on $X^r_{\mathfrak{h}}$,
where the second structure is induced by the natural identification
$A^G\simeq A^{\widetilde G}$. On the first algebra
$A^G$ its bracket is induced by the Poisson structure
$\eta$ defined on the whole space $X_{(H)}$.
A question arises: is there some Poisson structure on the manifold
$X^r_\mathfrak{h}$ which induces the above mentioned bracket on
$A^{\widetilde G}$. We will prove that  such a Poisson
structure exists if the Poisson structure
$\eta$ on $X$ is nondegenerate, i.e.
$\eta=\omega^{-1}$, where $\omega$ is  some
$G$-invariant symplectic structure on
$X$.

As it follows from Lemma~\ref{le.2} in this case the pair
$(X^r_\mathfrak{h}, \widetilde\omega)$, where,
$\widetilde \omega=i^*\omega$ and
$i: X^r_\mathfrak{h}\to X$ is the natural embedding, is a
symplectic manifold (the restriction $\omega(x)|_{T_x X^r_\mathfrak{h}}$
is nondegenerate for all $x\in X^r_\mathfrak{h}$). For any function
$f\in A^G$ its Hamiltonian vector field
${\mathcal{H}}_f$ is tangent to the submanifold
$X^r_\mathfrak{h}$ at each point
$x\in X^r_\mathfrak{h}$. This easily follows from the fact
that $df(\xi_X)=0$ for all
$\xi\in{\mathfrak{g}}$ and, in particular, for all
$\xi\in\mathfrak{p}$, i.e.
$\omega(x)({\mathcal{H}}_f(x),
{\mathcal{P}}(x)){\stackrel{\mathrm{def}}{=}}
-df(x)({\mathcal{P}}(x))=0$.
As $T_x X^r_\mathfrak{h}$ is a skew-orthogonal complement to
${\mathcal{P}}(x)$ by Lemma~\ref{le.2} we conclude that
${\mathcal{H}}_f(x)\in T_x X^r_\mathfrak{h}$.
Therefore, for any $x\in X^r_\mathfrak{h}$ and any vector field
$Y$ tangent to $X^r_\mathfrak{h}$ we have
$$
-d(i^*f)(x)(Y(x)){=}-df(x)(Y(x))=
\omega(x)({\mathcal{H}}_f(x),Y(x))=
\widetilde\omega(x)({\mathcal{H}}_f(x),Y(x)),
$$
i.e. the vector field
${\mathcal{H}}_f|_{X^r_\mathfrak{h}}$
is the hamiltonian vector field of the function
$i^*f$ with respect to the form
$\widetilde\omega$. Moreover, for any functions
$f_1,f_2\in A^G$ at
$x\in X^r_\mathfrak{h}$ we get the equality
\begin{equation}\label{eq.13}
\begin{split}
\eta(x)(df_1(x),&\, df_2(x))
{\stackrel{\mathrm{def}}{=}}\omega(x)({\mathcal{H}}_{f_2}(x),
{\mathcal{H}}_{f_1}(x)) \\
&=\widetilde\omega(x)({\mathcal{H}}_{i^*f_2}(x),
{\mathcal{H}}_{i^*f_1}(x))
=\widetilde\eta(x)(d(i^*f_1)(x), d(i^*f_2)(x)),
\end{split}
\end{equation}
where $\widetilde\eta$ is the Poisson structure
$\widetilde\omega^{-1}$ on $X^r_\mathfrak{h}$.

A {\em pair}
$(\eta_1,\eta_2)$ of linearly independent bi-vector fields
(bi-vectors for short) on a manifold
$X$ is called {\em Poisson} if
$\eta^t{\stackrel{\mathrm{def}}{=}} t_1\eta_1+t_2\eta_2$
is a Poisson bi-vector for any
$t=(t_1,t_2)\in {\mathbb R}^2$, i.e. each bi-vector
$\eta^t$ determines on
$X$ a Poisson structure with the Poisson bracket
$\{,\}^{t}:(f_1,f_2)\mapsto \eta^t(df_1,df_2)$;
the whole family of Poisson bi-vectors
$\{\eta^t\}_{t\in{\mathbb R}^2}$
is called a {\it bi-Poisson structure}. Remark here that a pair
$(\eta_1,\eta_2)$ of linearly independent Poisson structures
is Poisson if and only if
$t_1\eta_1+t_2\eta_2$ is a Poisson bi-vector for some
$(t_1,t_2)\in {\mathbb R}^2$ nonproportional to
$(1,0),(0,1)$. Indeed, the bi-vector
$\eta^t$ is Poisson if and only if
$[\eta^t,\eta^t]_S=0$, where $[,]_S$
is the so-called Schouten bracket \cite[\S 10.6]{MR}. The last
equation is quadratic with respect to
$t_1\colon t_2$.

A bi-Poisson structure
$\{\eta^t\}$ (we shall often skip the parameter space) can
be viewed as a two-dimensional vector space of Poisson
bi-vectors, the Poisson pair
$(\eta_1,\eta_2)$ as a basis in this space. Obviously, if
the Poisson structures
$\eta_1$ and $\eta_2$ are $G$-invariant,
then these structures induce a bi-Poisson structure on the manifold
$\mathbf X=X_{(H)}/G\simeq X^r_\mathfrak{h}/(N(H^0)/H^0)$
and, consequently, linear families of brackets on the spaces
$\mathcal{E}(\mathbf X)$, $A^G$, and
$A^{\widetilde G}$. The theorem below asserts that in a
particular case the linear family of brackets on the space
$A^{\widetilde G}$ is induced by some canonically defined
bi-Poisson structure on the manifold
$X^r_\mathfrak{h}$. Note that the action of the group
$\widetilde G= N(H^0)/H^0$ on
$X^r_\mathfrak{h}$ is locally free.

\begin{theorem}\label{th.4}
Let $\eta_1=\omega_1^{-1}$ and $\eta_2=\omega_2^{-1}$, where
$\omega_1,\omega_2$ are some
$G$-invariant symplectic forms on
$X$. Assume that the Poisson structures $\eta_1$ and
$\eta_2$ determine a bi-Poisson structure on
$X$. If the forms $\widetilde\omega_1=i^*\omega_1$ and
$\widetilde\omega_2=i^*\omega_2$ are linearly independent on
$X^r_\mathfrak{h}$ (here $i: X^r_\mathfrak{h}\to X$
is the natural embedding), then the Poisson structures
$\widetilde\eta_1=\widetilde\omega_1^{-1}$ and
$\widetilde\eta_2=\widetilde\omega_2^{-1}$ determine a
$\widetilde G$-invariant bi-Poisson structure on
$X^r_\mathfrak{h}$. This bi-Poisson structure induces on the
space $A^{\widetilde G}=A^G$ the same linear family of brackets as
the bi-Poisson structure induced by the pair
$(\eta_1,\eta_2)$ on the space
$A^G$. The action of the group $\widetilde G=N(H^0)/H^0$ on
$X^r_\mathfrak{h}$ is locally free.
\end{theorem}
\begin{proof}
It is sufficient to perform local reasoning.
Fix some point $x\in X^r_\mathfrak{h}$ and consider in $X$
the coordinate system
$(U(x), y_1,\ldots,y_p,y_{p+1},\ldots,y_{m})$
around the point $x$
as in Lemma~\ref{th.3}. Then in these coordinates the
symplectic forms $\omega_a$, $a=1,2$ are described by the
skew-symmetric matrices
$$
W_a(y(z))=\left(
\begin{array}{cc}
 A_a(y(z))  & B_a(y(z)) \\
C_a(y(z))   & D_a(y(z))
\end{array}\right), \quad\text{for}\ z\in U(x),
$$
such that
\begin{equation}\label{eq.14}
W_a(y(z))=
\left(
\begin{array}{cc}
 A_a(y(z))  & 0 \\
0   & D_a(y(z))
\end{array}\right),  \quad\text{for}\ z\in U_\mathfrak{h}(x)\subset U(x).
\end{equation}
Recall that $U_\mathfrak{h}(x)=U(x)\cap X^r_\mathfrak{h}$ and
$y_1(z)=\ldots=y_p(z)=0$ if
$z\in U_\mathfrak{h}(x)$. By the definition, the Poisson
structure $\eta^t=t_1\eta_1+t_2\eta_2$ is determined by the
$m\times m$-matrix
$t_1 W_1^{-1}(y)+t_2 W_2^{-1}(y)$:
$$
\eta^t(y)=\sum_{1\leqslant i<j\leqslant m} \bigl(t_1
W_1^{-1}+t_2 W_2^{-1}\bigr)_{ij}(y) \frac{\partial}{\partial
y_i}\land \frac{\partial }{\partial y_j}.
$$
Since the Poisson structures
$\eta_1,\eta_2$ are nondegenerate, then for some
$t=(t_1,t_2)\in {\mathbb R}^2\setminus(({\mathbb R}\times
\{0\})\cup(\{0\}\times{\mathbb R}))$
Poisson structure
$\eta^t$ is nondegenerate in each point of some open
neighborhood of the point
$x$, which we assume, without loss of generality, to be the
original open neighborhood
$U(x)$. Then the skew-symmetric matrix
$\bigl(t_1 W_1^{-1}(y)+t_2 W_2^{-1}(y)\bigr)^{-1}$
is a matrix of some symplectic form
$\omega^t$ on $U(x)$, i.e. the form
$$
\sum_{1\leqslant i<j\leqslant m} \Bigl(\bigl(t_1 W_1^{-1}
+t_2 W_2^{-1}\bigr)^{-1}\Bigr)_{ij}(y) dy_i\land dy_j
$$
is closed. Thus the form
$i^*\omega^t$, where
$i|_{U_\mathfrak{h}(x)}: U_\mathfrak{h}(x)\to U(x)$
is the embedding
$(y_{p+1},\ldots,y_m)\mapsto (0,\ldots,0,y_{p+1},\ldots,y_m)$
is also closed. By~(\ref{eq.14}) for points
$z\in U_\mathfrak{h}(x)$ with coordinates
$y=(0,\ldots,0,y_{p+1},\ldots,y_m)$ we have
$$
\bigl(t_1 W_1^{-1}
+t_2 W_2^{-1}\bigr)^{-1}(y)=
\mbox{\small
$\left(
\begin{array}{cc}
\bigl(t_1 A_1^{-1}
+t_2 A_2^{-1}\bigr)^{-1}(y)
 & 0 \\
0   & \bigl(t_1 D_1^{-1}
+t_2 D_2^{-1}\bigr)^{-1}(y)
\end{array}\right)$}.
$$
Taking into account that $y_1(z)=\ldots=y_p(z)=0$ on
the set $U_\mathfrak{h}(x)$ we obtain that the form
$$
(i^*\omega^t)(y)=\sum_{p+1\leqslant i<j\leqslant m}
\Bigl(\bigl(t_1 D_1^{-1}
+t_2 D_2^{-1}\bigr)^{-1}\Bigr)_{ij}(y) dy_i\land dy_j,
$$
where $y=(0,\ldots,0,y_{p+1},\ldots,y_m)$,
is closed. This means that the tensor
$\widetilde\eta^t=t_1\widetilde\eta_1+t_2\widetilde\eta_2$,
where $\widetilde\eta_1=\widetilde\omega_1^{-1}$ and
$\widetilde\eta_2=\widetilde\omega_2^{-1}$,
determines nondegenerate Poisson structure in the open subset
$U_\mathfrak{h}(x)\subset X^r_\mathfrak{h}$. Since
$(t_1,t_2)\not\in ({\mathbb
R}\times\{0\})\cup(\{0\}\times{\mathbb R})$,
the Poisson structures
$\widetilde\eta_1$ and $\widetilde\eta_2$ determine a
$\widetilde G$-invariant bi-Poisson structure on
$X^r_\mathfrak{h}$.

By~(\ref{eq.13}) the bracket on
the space $A^G$ at the point
$x$ induced by the Poisson structure
$\eta_a(x)$, $a=1,2$, i.e. by the symplectic structure
$\omega_a(x)$, coincides with the bracket induced by the Poisson
structure $\widetilde\eta_a(x)$. By linearity
the brackets on the space $A^G$ at the point
$x$ induced by Poisson structures $\eta^t(x)$ and
$\widetilde\eta^t(x)$ coincide for each
$t\in {\mathbb R}^2$.
\qed
\end{proof}

\section{Reduction of a bi-Poisson structure on the cotangent bundle of
the adjoint orbit of a compact Lie group}
\label{s.3}

Here we calculate the reduced bi-Poisson structure
on the manifold $X^r_\mathfrak{h}$ in the case when
$X$ is the cotangent bundle of
the adjoint orbit of a compact Lie group $G$ with an
invariant bi-Poisson structure constructed
in our paper~\cite{MP}.

Let $G$ be a compact connected Lie group with the Lie algebra
${\mathfrak{g}}$. Denote by $\langle\cdot ,\cdot\rangle$
an $\operatorname{Ad}(G)$-invariant scalar product on ${\mathfrak{g}}$.
Let ${\mathcal{O}}\subset{\mathfrak{g}}$ be the
$\operatorname{Ad}(G)$-orbit through some element
$a\in{\mathfrak{g}}$ of the Lie algebra
${\mathfrak{g}}$.  Then
${\mathcal{O}}=G/K$, where
$$
K=\{g\in G: \operatorname{Ad}(g)(a)=a\}
$$
is the isotropy group of $a$ (a connected closed subgroup of
$G$~\cite[Lemma 5]{Ko}). Denote by
$\Omega$ the canonical symplectic form on the cotangent
bundle $T^*{\mathcal{O}}$. The scalar product
$\langle\cdot,\cdot\rangle$ defines a $G$-invariant
metric on $G/K$. This metric identifies the cotangent
bundle $T^*{\mathcal{O}}$ and the tangent bundle $T{\mathcal{O}}$.
Thus we can also talk about the canonical 2-form
$\Omega$ on $T{\mathcal{O}}$. The symplectic form $\Omega$ is
$G$-invariant with respect to the natural action of $G$
on $T{\mathcal{O}}$ (extension of the action of $G$ on ${\mathcal{O}}$).

Let $\pi: T{\mathcal{O}}\to {\mathcal{O}}$
be the canonical projection. The orbit
${\mathcal{O}}\subset{\mathfrak{g}}$
is a symplectic manifold with the Kirillov-Kostant-Souriau form
$\omega$ (here we identified the reductive Lie
algebra ${\mathfrak{g}}$ with its dual space
${\mathfrak{g}}^*$ using the invariant scalar product
$\langle\cdot ,\cdot\rangle$ on
${\mathfrak{g}}$). So we can consider the closed
$G$-invariant 2-form $\Omega+\pi^*\omega$ on
$T{\mathcal{O}}$. This is a symplectic form on the manifold
$X=T{\mathcal{O}}$~\cite[Prop. 1.6]{MP}. Put $\omega_1=\Omega$ and
$\omega_2=\Omega+\pi^*\omega$. Write
$\eta_1=\omega_1^{-1}$, $\eta_2=\omega_2^{-1}$
for the inverse Poisson bi-vectors. The pair of Poisson structures
$(\eta_1,\eta_2)$ determines an
$G$-invariant bi-Poisson structure
$\{\eta^t=t_1\eta_1+t_1\eta_1\}$,
$(t_1,t_2)\in{\mathbb R}^2$, on
$X$ and the Poisson structure
$\eta^t$ is degenerate if and only if
$t_1+t_2=0$~\cite[Prop. 1.6]{MP}.

Let $\hat G$ be any connected closed Lie subgroup
of $G$ with the Lie algebra $\hat{\mathfrak g} \subset\mathfrak g$
containing the element $a$. Let $\hat {\mathcal{O}}$ be
the adjoint orbit through the element $a\in\hat{\mathfrak g}$
in the Lie algebra $\hat {\mathfrak{g}}$.
This orbit is a suborbit of ${\mathcal{O}}$,
i.e. $\hat{\mathcal{O}}=\operatorname{Ad}(\hat G)(a)\subset{\mathcal{O}}$.
Therefore $\hat{\mathcal{O}}=\hat G/\hat K$,
where $\hat K=K\cap \hat G$.
Denote by $j: T \hat{\mathcal{O}}\to T {\mathcal{O}}$
the natural embedding.

\begin{lemma}\label{le.5}
Let $\hat G$ be any connected closed Lie subgroup of
$G$ with the Lie algebra
$\hat{\mathfrak g}\subset \mathfrak g$ containing the element $a$. The restrictions
$\tilde\omega_1=\omega_1|_{T \hat{\mathcal{O}}}=j^*\omega_1$
and $\tilde\omega_2=\omega_2|_{T \hat{\mathcal{O}}}=j^*\omega_2$
are symplectic forms on the tangent bundle
$T \hat{\mathcal{O}}\subset T {\mathcal{O}}$.
The Poisson structures
$\tilde\eta_1=\tilde\omega_1^{-1}$ and
$\tilde\eta_2=\tilde\omega_2^{-1}$ determine a
$\hat G$-invariant bi-Poisson structure
$\{\tilde\eta^t=t_1\tilde\eta_1+t_1\tilde\eta_1\}$,
$(t_1,t_2)\in{\mathbb R}^2$, on
$T \hat{\mathcal{O}}$.
\end{lemma}
\begin{proof}
The restriction of the scalar product
$\langle\cdot,\cdot\rangle$ to the subalgebra $\hat {\mathfrak g}$
defines a $\hat G$-invariant
metric on $\hat G/\hat K$. This metric identifies the cotangent
bundle $T^*\hat {\mathcal{O}}$ and the tangent bundle $T\hat {\mathcal{O}}$.
Denote by $\hat\Omega$ the canonical 2-form on $T\hat{\mathcal{O}}$.
By~\cite[Prop. 4]{My16} the canonical form $\hat\Omega$
coincides with the restriction $\Omega|_{T\hat{\mathcal{O}}}$
of the canonical form $\Omega$, i.e.
$\hat\Omega=\tilde\omega_1$.

Identifying the compact Lie algebra
$\hat{\mathfrak{g}}$ with its dual space
$\hat{\mathfrak{g}}^*$ by means of the restriction of the
invariant scalar product
$\langle\cdot ,\cdot\rangle$ to
$\hat{\mathfrak{g}}$ we can say about the Kirillov-Kostant-Souriau
form $\hat\omega$ on the orbit
$\hat{\mathcal{O}}\subset \hat{\mathfrak{g}}$.
Let us show that $\hat\omega=\omega|_{\hat{\mathcal{O}}}$.

Indeed, by definition the form
$\omega$ is $G$-invariant and at the point $a\in{\mathcal{O}}$ we
have
$$
\omega(a)([a,\xi_1],[a,\xi_2])=-\langle
a,[\xi_1,\xi_2]\rangle, \quad \forall\xi_1,\xi_2\in{\mathfrak g},
$$
where we consider the vectors $[a,\xi_1],[a,\xi_2]\in
T_a{\mathfrak g}={\mathfrak g}$ as tangent vectors to the orbit
${\mathcal{O}}\subset{\mathfrak g}$ at the point $a\in{\mathcal{O}}$.
Since the form $\hat\omega$ is described by the similar relation
on the Lie algebra $\hat{\mathfrak g}$ containing the element
$a$, we obtain that $\omega(a)|_{T_a\hat{\mathcal{O}}}=\hat\omega(a)$.
Thus $\hat\omega=\omega|_{\hat{\mathcal{O}}}$ by the
$\hat G$-invariance of the forms $\hat\omega$ and  $\omega$.

Let $\hat\pi: T\hat{\mathcal{O}}\to \hat{\mathcal{O}}$
be the canonical projection. Consider the closed
$\hat G$-invariant 2-form
$\hat\Omega+\hat\pi^*\hat\omega$ on
$T\hat{\mathcal{O}}$. As above, the pair of the
$\hat G$-invariant symplectic forms
$\hat\omega_1=\hat\Omega$ and
$\hat\omega_2=\hat\Omega+\hat\pi^*\hat\omega$ on
$T\hat {\mathcal{O}}$ determines an
$\hat G$-invariant bi-Poisson structure  by~\cite[Prop. 1.6]{MP}.
Taking into account that
$\hat\pi^*\hat\omega=\hat\pi^*(\omega|_{\hat{\mathcal{O}}})
=(\pi^*\omega)|_{T\hat{\mathcal{O}}}$
and, consequently, $\hat\omega_k=\tilde\omega_k$, $k=1,2$,
we complete the proof.
\qed
\end{proof}

By the lemma above the $G$-invariant bi-Poisson structure
$\{\eta^t=t_1\eta_1+t_1\eta_1\}$,
$(t_1,t_2)\in{\mathbb R}^2$, on $T{\mathcal{O}}$ determines
the $\hat G$-invariant bi-Poisson structure
$\{\tilde\eta^t=t_1\tilde\eta_1+t_1\tilde\eta_1\}$,
$(t_1,t_2)\in{\mathbb R}^2$, on
$T\hat{\mathcal{O}}$. In general the natural embedding
$j: T\hat{\mathcal{O}}\to T{\mathcal{O}}$
is not a Poisson map w.r.t. the Poisson structures $\tilde\eta^t$ and $\eta^t$, i.e. $j^*:(\mathcal E(T{\mathcal{O}}),\{,\}_{\eta^t})\to (\mathcal E(T\hat{\mathcal{O}}),\{,\}_{\tilde\eta^t})$ is not a Lie algebra homomorphism. Moreover, the restriction $j^*|_{A^G}:(A^G,\{,\}_{\eta^t})\to (\mathcal E(T\hat{\mathcal{O}}),\{,\}_{\tilde\eta^t})$ to the space $A^G$ of the $G$-invariant functions on $T\mathcal O$ is not a Lie algebra homomorphism too.
However, using Theorem \ref{th.4}
we are able to  describe some subgroup
$\hat G\subset G$ and the corresponding orbit
$\hat {\mathcal{O}}=\hat G/\hat K$
for which the map $j^*|_{A^G}$ is a Lie algebra homomorphism (for any $t$), its image lies in the space $A^{\hat G}$ of $\hat G$-invariant functions on $T\hat{\mathcal O}$ and the
action of the group $\hat G/C(\hat G)$ on
$T\hat{\mathcal{O}}$ is locally free (see Proposition \ref{pr.6} below). Here
$C(\hat G)$ stands for the center of the Lie group
$\hat G$ (which is the kernel of the adjoint representation
of $\hat G$).

Let us describe the corresponding subgroups starting from the subgroup
$H\subset G$ determining the principal orbit type submanifold
$X_{(H)}$ of the $G$-manifold $X=T(G/K)$.
As we remarked above in this case the manifold
$X_{(H)}$ is a connected open dense subset of $X$. Denote by
${\mathfrak{k}}$ the Lie algebra of $K$ and by
${\mathfrak{m}}$ the orthogonal complement to ${\mathfrak{k}}$ in
${\mathfrak{g}}$ with respect to the form
$\langle\cdot,\cdot\rangle$. Taking into account that
$G$ acts on the base ${\mathcal{O}}\subset {\mathfrak{g}}$
transitively and identifying the tangent space
$T_o{\mathcal{O}}$ at $o=K$ with the space
${\mathfrak{m}}$, we obtain that
\begin{equation}\label{myeq}
H=\{k\in K: \operatorname{Ad}(k)(x_0)=x_0\}=K_{x_0}
\end{equation}
for some $x_0\in{\mathfrak{m}}$ such that the centralizer
${\mathfrak{k}}_{x_0}{\stackrel{\mathrm{def}}{=}}
\{y\in{\mathfrak{k}}: [x_0,y]=0\}$
has the minimal possible dimension. It is clear that the Lie algebra
$\mathfrak{h}$ of $H$ coincide with the Lie algebra
${\mathfrak{k}}_{x_0}$. Consider the compact Lie subalgebra
\begin{equation}\label{eq.15}
\hat{\mathfrak{g}}=\{y\in{\mathfrak{g}}:
[y,z]=0,\ \forall z\in\mathfrak{h}={\mathfrak{k}}_{x_0}\}
\end{equation}
of ${\mathfrak g}$.
Denote by
$\hat G$ the connected Lie subgroup of
$G$ with the Lie algebra $\hat{\mathfrak{g}}$.
The Lie group $\hat G$ is closed in $G$ because
$\hat G$ is the identity component of the centralizer of
$H^0$ in $G$. Moreover, $a$ is an element of
$\hat{\mathfrak g}$ because by definition
$[a,{\mathfrak k}]=0$ and
${\mathfrak h}\subset{\mathfrak k}$.
Thus, as above, we can consider $\operatorname{Ad}(\hat G)$-suborbit
$\hat{\mathcal{O}}\subset\hat{\mathfrak g}$ of
the orbit ${\mathcal{O}}$ through the element $a$ and
the natural embedding $j: T\hat{\mathcal{O}}\to T{\mathcal{O}}$.
\begin{proposition}\label{pr.6}
Let $\hat G$ be the connected Lie subgroup of
$G$ with the Lie algebra
$\hat{\mathfrak g}$ defined by~$(\ref{eq.15})$. Then
\begin{list}{}{\listparindent 0pt \itemsep 0pt
\parsep  0pt \topsep 0pt}
\item[(1)]
the restrictions
$\tilde\omega_1=\omega_1|_{T \hat{\mathcal{O}}}=j^*\omega_1$
and $\tilde\omega_2=\omega_2|_{T \hat{\mathcal{O}}}=j^*\omega_2$
are symplectic forms on the tangent bundle
$T \hat{\mathcal{O}}\subset T {\mathcal{O}}$;
\item[(2)]
the Poisson structures
$\tilde\eta_1=\tilde\omega_1^{-1}$ and
$\tilde\eta_2=\tilde\omega_2^{-1}$ determine a
$\hat G$-invariant bi-Poisson structure
$\{\tilde\eta^t=t_1\tilde\eta_1+t_1\tilde\eta_1\}$,
$(t_1,t_2)\in{\mathbb R}^2$, on
$T \hat{\mathcal{O}}$;
\item[(3)]
for any $t\in{\mathbb R}$ the map $j^*$
is a Poisson map of the $\eta^t$-Poisson algebra
$A^{G}$ of the $G$-invariant functions on $T {\mathcal{O}}$
into the $\tilde\eta^t$-Poisson algebra $A^{\hat G}$
of the $\hat G$-invariant function on $T\hat {\mathcal{O}}$;
\item[(4)]
the action of the Lie group $\hat G/C(\hat G)$
on $T\hat {\mathcal{O}}$ is locally free
($C(\hat G)$ is the center of $\hat G$);
\item[(5)]
the map $j^*:A^{G}\to A^{\hat G}$ is an injection
and the image $j^*(A^{G})$ functionally generates
the space $A^{\hat G}$.
\end{list}
\end{proposition}
\begin{proof}
Items (1) and (2) follow immediately from Lemma~\ref{le.5}.
To prove items (3)--(5)
we will describe the submanifolds
$X_H$ and $X^r_\mathfrak{h}$ of $X_{(H)}$ defined by
relations~(\ref{eq.2}),~(\ref{eq.5}),~(\ref{eq.6})
and will show that some connected component of $X^r_\mathfrak{h}$
is open and dense in $T\hat{\mathcal{O}}$.
To this end  we will use some calculation from
the paper~\cite[sect. 2.1, 3.3]{MP}.

It is clear that
$\hat{\mathcal{O}}=\hat G/\hat K$, where
$\hat K=\hat G\cap K$.
Since the form
$\langle\cdot ,\cdot \rangle$ is
$\operatorname{Ad}(G)$-invariant, we have
$[{\mathfrak{k}},{\mathfrak{m}}]\subset{\mathfrak{m}}$ and
$\operatorname{ad}(x_0)({\mathfrak{k}})\subset{\mathfrak{m}}$, where $x_0$ is that mentioned in formula (\ref{myeq}).
Let
$$
{\mathfrak{m}}(x_0)=\{y\in{\mathfrak{m}}:
\langle y,\operatorname{ad}(x_0)({\mathfrak{k}}) \rangle=0\}.
$$
By the $\operatorname{Ad}(G)$-invariance of
$\langle\cdot ,\cdot\rangle$, we have that
$x_0\in{\mathfrak{m}}(x_0)$.  The Lie group $K$ is
compact, hence by Remark \ref{re.7} below,
\begin{equation}\label{eq.16}
\operatorname{Ad}(K)({\mathfrak{m}}(x_0))={\mathfrak{m}}
\quad\text{and, consequently,}\quad
G\cdot ({\mathfrak{m}}(x_0))=T(G/K),
\end{equation}
i.e. each $G$-orbit in $T(G/K)$ intersects the linear subspace
${\mathfrak{m}}(x_0)\subset {\mathfrak{m}}=T_o(G/K)$.

\begin{remark}\label{re.7}
Relations~(\ref{eq.16}) hold if
$x_0$ is replaced by an arbitrary element
$x\in{\mathfrak{m}}$. This follows easily from the fact that
for any $y\in{\mathfrak{m}}$ the function
$k\mapsto \langle y ,\operatorname{Ad}(k)(x)\rangle$
on the compact group $K$ attains its maximum value at some point
$k_y\in K$. Differentiating
$\langle y ,\operatorname{Ad}(k_y\exp t\xi)(x)\rangle$ with
$\xi\in{\mathfrak{k}}$, we obtain that
$\operatorname{Ad}(k_y^{-1})(y)\bot\operatorname{ad}(x)({\mathfrak{k}})$.
\end{remark}

Consider the
$\operatorname{Ad}(K)$-action of the compact Lie group
$K$ on ${\mathfrak{m}}$. The space
${\mathfrak{m}}(x_0)$ is the orthogonal complement to the
tangent space
$T_{x_0}(\operatorname{Ad}(K)(x_0))=
\operatorname{ad}(x_0)({\mathfrak{k}})$
of the orbit $\operatorname{Ad}(K)(x_0)\subset{\mathfrak{m}}$ at
$x_0$ in ${\mathfrak{m}}$~\cite[Th.2.3.28]{OR}. Hence some open neighborhood of
$x_0$ in the linear space
${\mathfrak{m}}(x_0)\ni x_0$ is a slice for
$\operatorname{Ad}(K)$-action at
$x_0$. Since the group $H=K_{x_0}$
represents the principal orbit type, the action of
$H$ on this open neighborhood of
$x_0$ and, consequently, on the whole linear space
${\mathfrak{m}}(x_0)$ is trivial, i.e.
\begin{equation}\label{eq.17}
\operatorname{Ad}(h)(x)=x
\quad\text{for all}\quad h\in H\ \text{and}\ x\in{\mathfrak{m}}(x_0),
\end{equation}
and, consequently,
\begin{equation}\label{eq.18}
[{\mathfrak{m}}(x_0),\mathfrak{h}]=0
\end{equation}
(see~\cite[Prop. 9]{My01} for another proof of identity~(\ref{eq.18})).
It is clear that
${\mathfrak{m}}(x_0)\cap X_H$ is an open dense subset of
${\mathfrak{m}}(x_0)$. Let $X^r_\mathfrak{h}$
be the submanifold of the connected manifold
$X_{(H)}$ defined by relations~(\ref{eq.5}) and~(\ref{eq.6}).
From~(\ref{eq.17}) and
the definitions of the manifolds
$X_{(H)}$, $X_{H}$ and
$X^r_\mathfrak{h}$ it follows easily that
\begin{equation}\label{eq.19}
{\mathfrak{m}}(x_0)\cap X_H
={\mathfrak{m}}(x_0)\cap X_{(H)}
={\mathfrak{m}}(x_0)\cap X^r_\mathfrak{h}.
\end{equation}

Let us show that
\begin{equation}\label{eq.20}
X_H=N(H)\cdot ({\mathfrak{m}}(x_0)\cap X_H)
\quad\text{and}\quad
X^r_\mathfrak{h}=N(H^0)\cdot ({\mathfrak{m}}(x_0)\cap X^r_\mathfrak{h}),
\end{equation}
where $N(H)$ is the normalizer of $N$ in $G$
and  $N(H^0)$ is the normalizer of the identity
component $H^0$ of $H$ in $G$.
Indeed, by~(\ref{eq.16}) each point of the manifold
$X_H$ has the form $g\cdot x$ for some $g\in G$,
$x\in {\mathfrak{m}}(x_0)$ and for this point $G_{g\cdot x}=H$.
However, $G_{g\cdot x}=gG_x g^{-1}$ and
by~(\ref{eq.17}) $H\subset G_x$. Therefore
$gHg^{-1}\subset H$. Since $gHg^{-1}$ is an open
subgroup of $H$ and the compact group $H$ has
a finite number of connected component,
$gHg^{-1}=H$, i.e. $g\in N(H)$.
Similarly, each point of the manifold
$X^r_\mathfrak{h}$ has the form $g\cdot x$ for some $g\in G$,
$x\in {\mathfrak{m}}(x_0)$ and for this point $(G_{g\cdot x})^0=H^0$.
Also $G_{g\cdot x}=gG_x g^{-1}$ and
by~(\ref{eq.17}) $H\subset G_x$. Then $H^0=gG_x^0 g^{-1}\subset G_x^0$,
and, consequently, $H^0=gG_x^0 g^{-1}=G_x^0$. Thus
$gH^0g^{-1}=H^0$, i.e. $g\in N(H^0)$.

Note that the subgroup $N(H^0)$ of $G$ is closed (compact) and
therefore contains only a finite number of connected
components, i.e. $|N(H^0)/(N(H^0))^0|<\infty$.
Since by Lemma~\ref{le.1} $X^r_\mathfrak{h}$ is an embedded
submanifold of $X_{(H)}$ and of $X=T(G/K)$, its connected
component $X^{r,x_0}_\mathfrak{h}$ containing $x_0$ has the form
\begin{equation}\label{eq.21}
X^{r,x_0}_\mathfrak{h}=(N(H^0))^0\cdot
({\mathfrak{m}}(x_0)\cap X^r_\mathfrak{h})
\end{equation}
 and (see~(\ref{eq.8}))
\begin{equation}\label{eq.22}
\mathbf X=X_{(H)}/G\simeq X^r_\mathfrak{h}/(N(H^0)/H^0)
\simeq X^{r,x_0}_\mathfrak{h}/((N(H^0))^{x_0}/H^0),
\end{equation}
where $((N(H^0))^{x_0}$ is the normalizer
of the component $X^{r,x_0}_\mathfrak{h}$ in the
group $N(H^0)$ (containing the connected
component $((N(H^0))^0$ of $N(H^0)$).
Since by definition $H\subset N(H^0)$, then
$h\cdot (N(H^0))^0\cdot h^{-1}=(N(H^0))^0$ for any $h\in H$.
Taking into account that $\operatorname{Ad}(H)(x)=x$
for each $x\in\mathfrak{m}(x_0)$, we obtain that
$$
H\subset (N(H^0))^{x_0}.
$$
Now it is clear that the manifold
$X^{r,x_0}_\mathfrak{h}$ is a single orbit type
$(N(H^0))^{x_0}/H^0$-manifold with a discrete
isotropy group isomorphic to $H/H^0$ (the group $H^0$
acts trivially on $X^{r,x_0}_\mathfrak{h}$).

We will show that the connected component
$X^{r,x_0}_\mathfrak{h}$ of the manifold
$X^r_\mathfrak{h}$ containing the element
$x_0$ is an open dense subset of
$T\hat {\mathcal{O}}$. To this end
consider  the subalgebra
$\hat{\mathfrak{k}}={\mathfrak k}\cap \hat{\mathfrak{g}}$
of ${\mathfrak k}$.
Since ${\mathfrak{k}}$ is the centralizer of $a\in{\mathfrak{g}}$
in ${\mathfrak{g}}$, the element
$a\in{\mathfrak{k}}$ belongs to $\hat{\mathfrak{k}}$
($[a,{\mathfrak{k}}_{x_0}]=[a,{\mathfrak{k}}]=0$).
Denote by
$\hat{\mathfrak{m}}$ the orthogonal complement to
$\hat{\mathfrak{k}}$ in
$\hat{\mathfrak{g}}$ with respect to the form
$\langle\cdot,\cdot\rangle|_{\hat {\mathfrak{g}}}$.
By~(\ref{eq.18})
${\mathfrak{m}}(x_0)\subset\hat{\mathfrak{m}}$. Moreover,
${\mathfrak{m}}(x_0)$ is the orthogonal complement of the
space $\operatorname{ad}(x_0)(\hat{\mathfrak{k}})$ in
$\hat{\mathfrak{m}}$~\cite[Prop.
2.3]{MP}, i.e. $\hat{\mathfrak{m}}(x_0)={\mathfrak{m}}(x_0)$.
Now applying Remark~\ref{re.7} to the pair
$(\hat G,\hat K)$ we get
$\hat G\cdot {\mathfrak{m}}(x_0)=T\hat{\mathcal{O}}$.

Since the compact Lie algebra
$\mathfrak{h}$ is reductive, we have that
$\mathfrak{n}(\mathfrak{h})=\hat{\mathfrak{g}}+\mathfrak{h}$
for the normalizer
$\mathfrak{n}(\mathfrak{h})$ of $\mathfrak{h}$ in
${\mathfrak{g}}$. Then $gh=hg$ for all elements
$g\in\hat G$ and $h\in H^0$ because
$\hat G$ is a connected component of the centralizer of
$H^0$ in $G$.
It is clear that $\hat G\cdot H^0\subset G$ is the identity component of the
normalizer $N(H^0)$. However,
$H^0\cdot {\mathfrak{m}}(x_0)={\mathfrak{m}}(x_0)$
by~(\ref{eq.17}), and therefore
\begin{equation}\label{eq.23}
\begin{split}
T\hat{\mathcal{O}}
&=\hat G\cdot {\mathfrak{m}}(x_0)=
(\hat G\cdot H^0)\cdot {\mathfrak{m}}(x_0)\\
&=(N(H^0))^0\cdot {\mathfrak{m}}(x_0)
=(N(H^0))^{x_0}\cdot {\mathfrak{m}}(x_0).
\end{split}
\end{equation}
Since by~(\ref{eq.21})
$(N(H^0))^0\cdot ({\mathfrak{m}}(x_0)\cap X^r_\mathfrak{h})$
is the connected component
$X^{r,x_0}_\mathfrak{h}$ of the manifold $X^r_\mathfrak{h}$,
$X^{r,x_0}_\mathfrak{h}$ is an open dense subset of
$\hat X=T\hat {\mathcal{O}}$. This subset is
$\hat G$-invariant because $\hat G\subset (N(H^0))^0$. But
$X_{(H)}\subset T{\mathcal{O}}$ and
$X^{r,x_0}_\mathfrak{h}\subset T\hat{\mathcal{O}}$.
Thus by~(\ref{eq.22}) and Theorem~\ref{th.4} for any
$t\in{\mathbb R}^2$ the map
$i^*=(j|_{X^{r,x_0}_\mathfrak{h}})^*$ is a Poisson map of the
$\eta^t$-Poisson algebra of the
$G$-invariant function on $X_{(H)}$ into the
$\tilde\eta^t$-Poisson algebra of the
$(N(H^0))^{x_0}/H^0$-invariant function on
$X^{r,x_0}_\mathfrak{h}$.
Now to prove item (3) it is sufficient to remark
that  $X_{(H)}$ and $X^{r,x_0}_\mathfrak{h}$ are
open and dense in $T{\mathcal{O}}$ and
$T\hat{\mathcal{O}}$ respectively.

By Lemma~\ref{le.1} the actions of the groups
$(N(H^0))^{x_0}/H^0$ and $(N(H^0))^0/H^0$ on
$X^{r,x_0}_\mathfrak{h}$ are locally free. As we remarked
above $X^{r,x_0}_\mathfrak{h}$ is a single orbit type
$(N(H^0))^{x_0}/H^0$-manifold with a discrete isotropy group
isomorphic to $H/H^0$. Therefore by~(\ref{eq.23})
$X^{r,x_0}_\mathfrak{h}$ is also a single orbit type
$\hat G$-manifold with the isotropy group isomorphic to
$\hat H=\hat G\cap H$ and $\hat H$
is a Lie group determining the principal orbit type for the
$\hat G$-action on $T\hat{\mathcal{O}}$.
Taking into account that
$[\hat{\mathfrak g},{\mathfrak h}]=0$ by definition,
we obtain that the Lie algebra
$\hat{\mathfrak g}\cap{\mathfrak h}$
is a subalgebra of the center of
$\hat{\mathfrak g}$ and, consequently,
$\hat G\cap H^0\subset C(\hat G)$, where
$C(\hat G)$ is the kernel of the adjoint representation of
$\hat G$. Thus $\hat G\cap H^0\subset C(\hat G)\cap H$.
Therefore the action of the group
$\hat G/C(\hat G)$ on
$X^{r,x_0}_\mathfrak{h}\subset T\hat {\mathcal{O}}$
with a discrete isotropy group isomorphic to some quotient group of
$(\hat G\cap H)/(\hat G\cap H^0)$
is locally free, item (4) is proved.

Since $(N(H^0))^0=\hat G\cdot H^0$ and $H^0$ acts trivially on
$X^{r,x_0}_\mathfrak{h}$, each connected component of the
$(N(H^0))^{x_0}$-orbit in $X^{r,x_0}_\mathfrak{h}$ is some
$\hat G$-orbit and, consequently, the natural projection
$X^{r,x_0}_\mathfrak{h}/\hat G\to X^{r,x_0}_\mathfrak{h}/(N(H^0))^{x_0}$
is a covering. Taking into account that
$X_{(H)}/G\simeq X^{r,x_0}_\mathfrak{h}/(N(H^0))^{x_0}$
(see~(\ref{eq.22})) we complete the proof of (5).
\qed
\end{proof}

\section{Appendix}

The goal of this section is to provide the reader
with a proof of a statement that is well known to experts but
does not seem to be readily available in the literature.

Let $G$ be a connected Lie group and
$H^0$ be its compact connected subgroup. Denote by
${\mathfrak{g}}$ and
$\mathfrak{h}$ the Lie algebras of $G$ and
$H^0$ respectively. Let
$N(H^0)$ be the normalizer group of $H^0$ in
$G$. The Lie algebra of $N(H^0)$ is the normalizer
$\mathfrak{n}(\mathfrak{h})$ of the algebra
$\mathfrak{h}$ in ${\mathfrak{g}}$.

\begin{lemma}\label{le.8}
Let $\alpha$, $\beta$ be two
$\operatorname{Ad}(H^0)$-invariant
scalar products on the algebra Lie
${\mathfrak{g}}$. Let $\mathfrak{p}^\alpha$ and
$\mathfrak{p}^\beta$ be the orthogonal complements to
$\mathfrak{n}(\mathfrak{h})$ in
${\mathfrak{g}}$ with respect to the forms
$\alpha$ and $\beta$ respectively. Then
$\mathfrak{p}^\alpha\oplus\mathfrak{h}=\mathfrak{p}^\beta\oplus\mathfrak{h}$.
\end{lemma}
\begin{proof}
The forms $\alpha,\beta$ determine the
$\operatorname{Ad}(H^0)$-invariant
scalar products on the quotient space
${\mathfrak{g}}/\mathfrak{h}$ which we denote by
$\alpha'$ and $\beta'$ respectively. Let
$\pi:{\mathfrak{g}}\to{\mathfrak{g}}/\mathfrak{h}$
be the natural projection. By definition, the spaces
$\pi(\mathfrak{p}^\alpha)$ and
$\pi(\mathfrak{p}^\beta)$ are the orthogonal complements to
the space $\pi(\mathfrak{n}(\mathfrak{h}))$ in
${\mathfrak{g}}/\mathfrak{h}$ with respect to the forms
$\alpha'$ and
$\beta'$ respectively. Since the scalar product
$\alpha'$ on ${\mathfrak{g}}/\mathfrak{h}$ is
$\operatorname{Ad}(H^0)$-invariant,
there exists a unique nondegenerate linear map
$J:{\mathfrak{g}}/\mathfrak{h}\to{\mathfrak{g}}/\mathfrak{h}$
such that $\beta'(u,v)=\alpha'(u,Jv)$ for all
$u,v\in {\mathfrak{g}}/\mathfrak{h}$ and
$J\cdot\operatorname{Ad}(h)=\operatorname{Ad}(h)\cdot J$
for all $h\in H^0$.

If $\xi\in\mathfrak{n}(\mathfrak{h})$ then
$\operatorname{Ad}(h)(\xi)-\xi\in\mathfrak{h}$ for any
$h\in H^0$. This follows from the fact that
$h\exp(t\xi)h^{-1}\exp(-t\xi)\in H^0$
($H^0$ is a normal subgroup of $N(H^0)$). Conversely, if
$\xi\in{\mathfrak{g}}$ and
$\operatorname{Ad}(h)(\xi)-\xi\in\mathfrak{h}$ for any
$h\in H^0$ then
$[\xi,\mathfrak{h}]\subset\mathfrak{h}$, i.e.
$\xi\in\mathfrak{n}(\mathfrak{h})$. In other words,
\begin{equation}\label{eq.24}
\mathfrak{n}(\mathfrak{h})=\{\xi\in {\mathfrak{g}}:
\operatorname{Ad}(h)(\xi)-\xi\in\mathfrak{h},\ \forall h\in H^0\}
\end{equation}
and
\begin{equation}\label{eq.25}
\pi(\mathfrak{n}(\mathfrak{h}))=
\mathfrak{n}(\mathfrak{h})/\mathfrak{h}
=\{v\in {\mathfrak{g}}/\mathfrak{h}:
\operatorname{Ad}(h)(v)=v,\ \forall h\in H^0\}.
\end{equation}
Now we get the inclusion
$J(\pi(\mathfrak{n}(\mathfrak{h})))\subset
\pi(\mathfrak{n}(\mathfrak{h}))$
due to the fact that $J$ commutes with the
$\operatorname{Ad}(H^0)$-action on
${\mathfrak{g}}/\mathfrak{h}$ and
$\pi(\mathfrak{n}(\mathfrak{h}))\subset {\mathfrak{g}}/\mathfrak{h}$
is the set of all $\operatorname{Ad}(H^0)$-fixed vectors in
${\mathfrak{g}}/\mathfrak{h}$. Therefore
$$
\beta'(\pi(\mathfrak{p}^\alpha),\pi(\mathfrak{n}(\mathfrak{h})))
=\alpha'(\pi(\mathfrak{p}^\alpha),J(\pi(\mathfrak{n}(\mathfrak{h})))
=\alpha'(\pi(\mathfrak{p}^\alpha),\pi(\mathfrak{n}(\mathfrak{h})))=0,
$$
and, consequently,
$\pi(\mathfrak{p}^\alpha)=\pi(\mathfrak{p}^\beta)$, i.e.
$\mathfrak{p}^\alpha\oplus\mathfrak{h}
=\mathfrak{p}^\beta\oplus\mathfrak{h}$.
\qed
\end{proof}


\begin{thebibliography}{99}
\bibitem[1]{Bou}
N.~Bourbaki, \emph{Groupes et alg\`ebres de Lie, {I--III}},
\'El\'ements de math\'ematique, Hermann, Paris {VI}, 1971,1972.

\bibitem[2]{DK}
J.J. Duistermaat and J.A.C. Kolk,
Lie groups, Universitext, Springer-Verlag, 2000.

\bibitem[3]{GS1}
Guillemin, V. and  Sternberg, S.,
\emph{Convexity properties of the moment map},
Invent. math., \textbf{67} (1982), 491--513.


\bibitem[4]{GS}
Guillemin, V. and  Sternberg, S.,
\emph{Symplectic Techniques in Physics},
Cambridge University Press, 1984.

\bibitem[5]{Ko}
B.~Kostant,
\emph{Lie group representation on polynomial rings},
Amer. J. Math. \textbf{85} (1963), no.~3, 327--404.

\bibitem[6]{MR}
Marsden J.E., Ratiu T.S.,
\emph{Introduction to mechnics and symmetry},
Springer, N.Y. Berlin, 1999.

\bibitem[7]{MP}
Mykytyuk I.V., Panasyuk A.,
Bi-Poisson structures and integrability of
geode\-sic flow on homogeneous spaces,
Transformation Groups,
\textbf{9} (2004), no.~3, 289--308.

\bibitem[8]{My01}
I.~V. Mykytyuk, \emph{Actions of Borel subgroups on
homogeneous spaces of reductive complex Lie groups and
integrability}, Compositio Math.,
\textbf{127} (2001), no.~1, 55--67.

\bibitem[9]{My16}
I.~V. Mykytyuk, \emph{Integrability
of geodesic flows for metrics on suborbits of the adjoint
orbits of compact groups}, Transformation groups,
\textbf{21} (2016), no.~2, 531--553.

\bibitem[10]{OR}
J-P. Ortega and T.S. Ratiu, Momentum maps and Hamiltonian reduction,
Prog. Math. 222, Birkh\"auser Boston, Inc.,
Boston, MA, pp. xxxiv+497 (2004).

\bibitem[11]{Pa}
A.~Panasyuk, \emph{Projections of Jordan bi-Poisson
structures that are Kronecker, diagonal actions, and the
classical Gaudin systems}, Journ.  Geom. Phys.
\textbf{47} (2003), 379--397.
\end{thebibliography}
\end{document}